\documentclass[a4paper,oneside,12pt,final]{article}

\usepackage{amsmath,amsfonts,amscd,amssymb,amsthm,dsfont}
\usepackage{longtable,geometry}
\usepackage[english]{babel}
\usepackage[utf8]{inputenc}
\usepackage[active]{srcltx}
\usepackage[T1]{fontenc}
\usepackage{graphicx}
\usepackage{pstricks}
\usepackage{enumerate}
\usepackage{nicefrac}
\usepackage{calrsfs}

 \usepackage{rotating}

\usepackage{framed}
\usepackage{xcolor}
\usepackage{tikz}
\usetikzlibrary{calc}
\usepackage{xifthen}
\usetikzlibrary{decorations.pathreplacing}

\usepackage{calc}

\usepackage{geometry}
\usepackage[colorinlistoftodos, textwidth=\marginparwidth]{todonotes}

\usepackage[colorlinks]{hyperref}
\hypersetup{final}
\usepackage{xcolor}
\hypersetup{colorlinks, linkcolor={blue!40!black}, citecolor={blue!40!black}, urlcolor={blue!80!black}}

\usepackage{xcolor}

\usepackage[notref, notcite]{showkeys}

\usepackage[colorlinks]{hyperref}
\usepackage[figure,table]{hypcap} 
\hypersetup{final}
\usepackage{xcolor}
\hypersetup{
    colorlinks,
    linkcolor={blue!50!black},
    citecolor={blue!50!black},
    urlcolor={blue!80!black}
}

\DeclareMathOperator\Var{Var}

\usepackage{courier}
\usepackage{enumerate}

\usepackage{array}
\newcolumntype{e}{>{\displaystyle}r @{\,} >{\displaystyle}c @{\,} >{\displaystyle}l}

  \newcounter{constant}
  \newcommand{\nc}[1]{\refstepcounter{constant}\label{#1}}
  \newcommand{\uc}[1]{c_{\textnormal{\tiny \ref{#1}}}}
  \newcounter{polyhomeo}
  \newcommand{\ncg}[1]{\refstepcounter{polyhomeo}\label{#1}}
  \newcommand{\ucg}[1]{g_{\textnormal{\tiny \ref{#1}}}}
\def\clap#1{\hbox to 0pt{\hss#1\hss}}

\usepackage{calc}

\def\arraypar#1{\parbox[c]{\textwidth - 2cm}{\centering #1}}

\usepackage{environ}
\NewEnviron{display}{
  \begin{equation}\begin{array}{c}
    \arraypar{\BODY}
  \end{array}\end{equation}
}

\colorlet{shadecolor}{blue!15}

\newtheorem{theorem}{Theorem}[section]

\newtheorem{lemma}[theorem]{Lemma}
\newtheorem{proposition}[theorem]{Proposition}

\newtheorem{remark}[theorem]{Remark}

\newcommand{\be}[1]{\begin{equation}\label{#1}}
\newcommand{\ee}{\end{equation}}
\numberwithin{equation}{section}

\newcommand{\ba}[1]{\begin{align}\label{#1}}
\newcommand{\ea}{\end{align}}
\numberwithin{equation}{section}

\newcommand{\ben}{\begin{equation*}}
\newcommand{\een}{\end{equation*}}
\numberwithin{equation}{section}

\newcommand{\bbE}{\mathbb{E}}

\newcommand{\bbN}{\mathbb{N}}

\newcommand{\bbP}{\mathbb{P}}

\newcommand{\bbZ}{\mathbb{Z}}

\newcommand{\sfC}{{\sf C}}

\newcommand{\ep}{\varepsilon}

\newcommand{\lr}[1][]{\stackrel{#1}\rightarrow}

\newcommand{\rk}[1]{\bgroup\color{red}%
  \par\medskip\hrule\smallskip%
  \noindent\textbf{#1}%
  \par\smallskip\hrule\medskip\egroup}

\title{The box-crossing property for critical two-dimensional oriented percolation}

\author{Duminil-Copin, H
\footnote{Institut des Hautes \'Etudes Scientifiques,
Le Bois-Marie 35, route de Chartres 91440 Bures-sur-Yvette France, {\itshape \texttt{duminil@ihes.fr}}, and
Département de mathématiques -- Université de Genève,
2-4 rue du lièvre, 64 1211 Genève 4, Switzerland}
\and Tassion, V. \footnote{Département de mathématiques -- Université de Genève,
2-4 rue du lièvre, 64 1211 Genève 4, Switzerland, {\itshape \texttt{vincent.tassion@unige.ch}}} \and Teixeira, A.
\footnote{Instituto Nacional de Matem\'atica Pura e Aplicada -- IMPA,
Estrada Dona Castorina 110, 22460-320, Rio de Janeiro RJ, Brazil, {\itshape \texttt{augusto@impa.br}}}
}

\date{\today}

\begin{document}
\maketitle

\begin{abstract}
  We consider critical oriented Bernoulli percolation on the square lattice $\bbZ^2$. We prove a Russo-Seymour-Welsh type result which allows us to
 derive several new results concerning the critical behavior:
  \begin{itemize}
  \item We establish that the probability that the origin is connected to distance $n$ decays polynomially fast in $n$.
  \item We prove that the critical cluster of $0$ conditioned to survive to distance $n$ has a typical width $w_n$ satisfying $\varepsilon n^{2/5}\le w_n\le n^{1-\ep}$ for some $\ep>0$.
  \end{itemize} The sub-linear polynomial fluctuations contrast with the supercritical regime where $w_n$ is known to behave linearly in $n$. It is also different from the critical picture obtained for non-oriented Bernoulli  percolation, in which the scaling limit is non-degenerate in both directions. All our results extend to the graphical representation of the one-dimensional contact process.
\end{abstract}

\section{Introduction}
\label{s:intro}

\subsection{Motivation}

Oriented percolation, which is a directed version of classical Bernoulli percolation (introduced by Broadbent and Hammersley \cite{BroHam57} to understand percolation of a liquid in a porous medium), provides a model for a variety of physical systems in chemistry, solid state physics, and astrophysics. At a theoretical level, it is one of the simplest system exhibiting a phase transition, and has been as such an objet of intensive study in the last fifty years. It is also related to the geometric representation of the one-dimensional contact process introduced by Harris \cite{Har74,Har78} and is therefore interesting from the point of view of particles systems as well. We refer to \cite{Dur84} for a review on the subject and for further references.

The model is defined as follows. Consider the rotated (and rescaled) square lattice $\mathbb L:=\{(x_1,x_2)\in\bbZ^2:x_1+x_2\text{ even}\}$. Each vertex $x\in\mathbb L$ is connected to the vertices $x + (-1, 1)$ and $x + (1, 1)$ by two oriented edges, see Fig.~\ref{oriented_lattice}. Let $p\in[0,1]$. Each oriented edge is said to be {\em open} with probability $p$, and {\em closed} with probability $1-p$, independently of the state of the other edges. The law of the set of open edges is denoted by $\bbP_p$.
   \begin{figure}
    \label{oriented_lattice}
    \centering
 \includegraphics[width=0.50\textwidth]{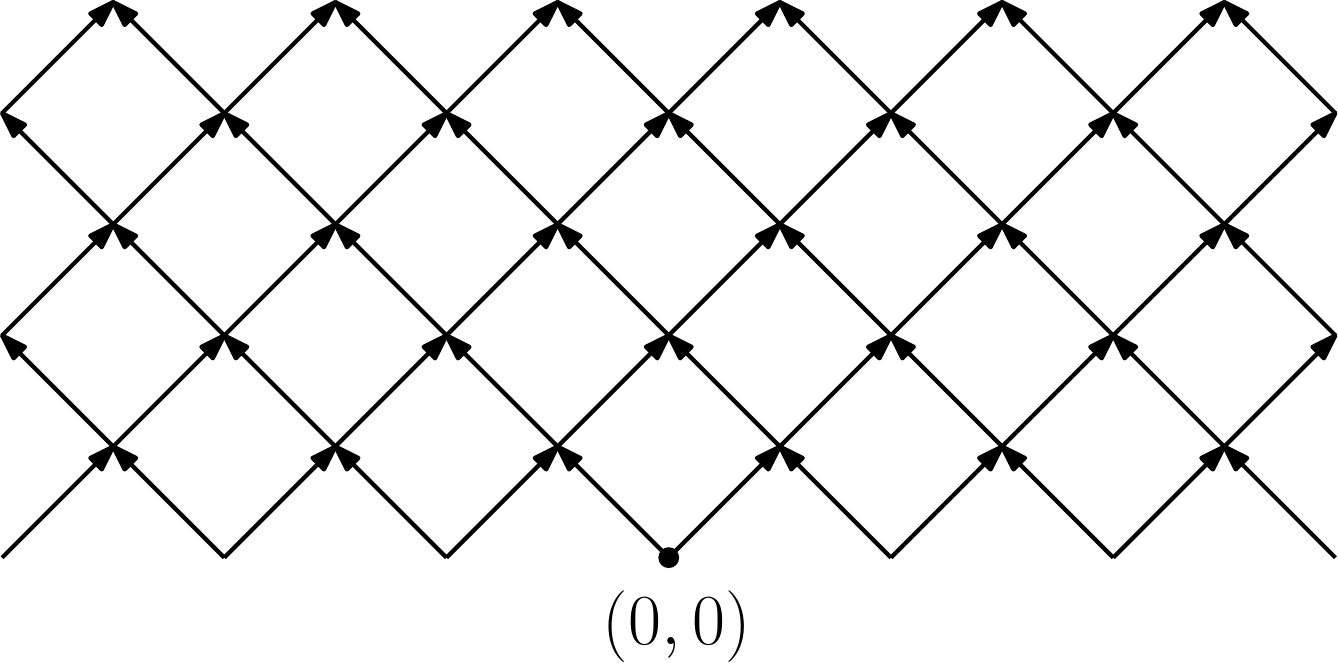}
    \caption{The lattice $\mathbb L$ with the oriented edges. }
  \end{figure}

  In oriented percolation, we study the connectivity properties of the random graph with vertex set $\mathbb L$, and edge set given by the open oriented edges. These open oriented edges should be understood as the set of edges allowing us to go upwards in the system. An {\em open path} is a collection of vertices $x_0,x_1,\dots,x_k$ such that the oriented edge $(x_i,x_{i+1})$ is open for every $0\le i<k$. Two vertices $x$ and $y$ are said to be {\em connected} (denoted $x\rightarrow y$) if there exists an open path starting at $x$ and ending at $y$. Let $\sfC_0$ be the connected component of the origin, i.e.~the set of vertices $x$ such that $0\rightarrow x$. In what follows, $0\rightarrow \infty$ denotes the event that $\sfC_0$ is infinite.

  One of the main interest of the model lies in the existence of a phase transition at a value $p_c\in(0,1)$ such that $\bbP_p(0\rightarrow \infty)=0$ if $p<p_c$, and above which $\bbP_p(0\rightarrow\infty)>0$ if $p>p_c$ (see \cite{BalBolSta94,BelLog06} for non-trivial lower and upper bounds on $p_c$). For $p<p_c$, connectivity properties are known to decay exponentially fast (see \cite{Gri81} for the original proof, and \cite{Dur84} for more details), while for $p>p_c$, the global shape of $\mathsf C_0$ converges to a cone of opening $\alpha(p)>0$ and Gaussian fluctuations on the boundary of $\mathsf C_0$, as proved in \cite{GalPre87,Kuc89}. An alternative proof of exponential decay for $p<p_c$ together with a proof of the mean-field bound $\bbP_p(0\rightarrow\infty)\ge c(p-p_c)$ were provided recently in \cite{DumTas15}. These results are just a few examples illustrating the more general motto that the subcritical and supercritical phases $p<p_c$ and $p>p_c$ are now well understood.

In \cite{DurGri83} and in \cite{BezGri90} respectively, the authors proved that $\alpha(p_c)=0$ and $\mathbb P_{p_c}[0\rightarrow \infty]=0$. These results naturally raise the question of quantitative bounds on the probability of being connected to distance $n$ and the typical width of large connected components at criticality.
In this paper, we
 provide polynomial upper bounds on these quantities (some lower bounds were proved previously in \cite{DurSchTan89}).

\subsection{Main results}

The main results of this paper deal with the critical phase $p=p_c$. The first theorem states that the probability that 0 is connected to distance $n$ decays polynomially fast. For $n\ge0$, define $\ell_n:=\mathbb Z\times\{n\}$.

\begin{theorem}\label{thm:1}
There exists $\varepsilon>0$ such that for every $n\ge1$,
$$\frac{\varepsilon}{n^{1/5}}\le \bbP_{p_c}(0\rightarrow\ell_n)\le \frac1{n^{\varepsilon}}.$$
\end{theorem}
The lower bound $\bbP_{p_c}(0\rightarrow \ell_n)\ge \frac{\varepsilon}{n^{1/4}}$ for all $n$ was proved in \cite{DurSchTan89}. Furthermore, the bound $\bbP_{p_c}(0\rightarrow \ell_n)\ge \frac{\varepsilon}{n^{1/5}}$ was also derived for infinitely many scales. We rely on the argument in \cite{DurSchTan89} for the lower bound. The novelty of this paper lies in the upper bound.

The second theorem deals with the typical width of the set of vertices connected to the origin. More precisely, let
$$R_n:=\max\{x\in\mathbb Z:\exists y\le 0\text{ even such that }(y,0)\rightarrow(x,n)\}.$$
Note that when $0\rightarrow \ell_n$, then $R_n$ is the first coordinate of the right-most point of $\sfC_0$. In some sense, the quantity $R_n$ can be understood as the width of a typical cluster that reaches distance $n$. The next theorem provides non-trivial polynomial bounds on $R_n$.
\begin{theorem}
\label{thm:2}
 \nc{c:2}
  There exists $\varepsilon>0$ such that for every $n\ge1$,
  \begin{equation}
    \label{eq:10}
    \varepsilon n^{2/5}\le \bbE_{p_c} ( R_n\,|\,0\lr  \ell_n)\le n^{1-\varepsilon}.
  \end{equation}
\end{theorem}

Again, the lower bound $\bbE_{p_c} ( R_n\,|\,0\lr \ell_n)\ge \varepsilon n^{1/2}$ for all $n$ (and $\varepsilon n^{2/5}$ for infinitely many scales) was proved in \cite{DurSchTan89}. The novelty of the paper lies in the upper bound. We wish to highlight the fact that the existence of $\varepsilon>0$ in the $n^{1-\varepsilon}$ upper bound is maybe the most important feature of the previous theorem. It implies that large connected components are rather thin. We should mention that it was shown that $\alpha(p)\searrow 0$ as $p\searrow p_c$, thus suggesting that the scaling limit indeed needs to be rescaled differently in the $x$ and $y$ coordinates (contrarily to the non-oriented cases where both directions play symmetric roles). The quantitative polynomial bound seems to be new.

We believe that the techniques developed to prove the two previous theorems should be very useful to study more delicate properties of the critical phase. In order to emphasize the technique, we isolate one important technical statement, called the {\em box-crossing property}, which we consider as one of the main new inputs of the paper.

The statement of the box-crossing property involves crossing probabilities. A {\em vertical crossing} of a box $B=[a,b]\times[c,d]$ is an open path of vertices {\em in }$B$ from the bottom $[a,b]\times\{c\}$ to the top $[a,b]\times\{d\}$ of $B$. A {\em left-right crossing} is an open path of vertices from the left $\{a\}\times[c,d]$ to the right $\{b\}\times[c,d]$ of $B$. Similarly, one define a {\em right-left crossing} of $B$. If such a vertical (resp.~left-right, right-left) crossing exists, we say that $B$ is {\em crossed vertically} (resp.~from left to right, from right to left). Define
\begin{align*}
V_p(m,n)&:=\bbP_p([0,m]\times[0,n]\text{ is crossed vertically}),\\
H_p(m,n)&:=\bbP_p([0,m]\times[0,n]\text{ is crossed from left to right}).
\end{align*}
By symmetry, $H_p(m,n)$ is also the probability that $[0,m]\times[0,n]$ is crossed from right to left. We are now ready to state our main technical statement.
\nc{c:box_crossing}
\begin{theorem}[the box-crossing property]
  \label{t:box_crossing}
  There exist a sequence of integers $(w_n)_{n\ge1}$ and a constant $\uc{c:box_crossing} > 0$ such that
  \begin{align}
    &\uc{c:box_crossing}\le H_{p_c}(3w_n,n)\le H_{p_c}(w_n, 3 n) \leq 1 - \uc{c:box_crossing}.\\
    & \uc{c:box_crossing}\le V_{p_c}(w_n,3n)\le V_{p_c}(3 w_n, n) \leq 1 - \uc{c:box_crossing}.
  \end{align}
\end{theorem}
We wish to highlight that similar statements are also available in the context of critical non-oriented percolation, with $w_n=n$ in this case.
We will see that $w_n$ is of the same order as $\bbE_{p_c} ( R_n\,|\,0\lr  \ell_n)$ and can therefore be intuitively understood as the typical width of a connected component of height $n$. Contrarily to the non-oriented case, we will show that $w_n$ is not growing linearly but is in fact smaller than $n^{1-\varepsilon}$.

The different rectangles involved in the previous statement will be the ``elementary bricks'' for all the constructions made in this article. The quantities on the right correspond to ``crossings in the easy direction'', while those on the left corresponds to ``crossings in the hard direction'', meaning that compared to a ``square box'' of size $w_n$ times $n$, the events on the left involve rectangles which are three times longer in the direction of crossing, while the events on the right involve rectangles which are three times larger orthogonally to the direction of crossing. The proof of the box-crossing property is based on an analog in the oriented case of the Russo-Seymour-Welsh (RSW) result for two-dimensional non-oriented Bernoulli percolation (see \cite{DumTas16} for a recent survey on this subject). This RSW result is stated as Theorem~\ref{t:RSW} in Section~\ref{s:rsw}. The reader should be careful that the specificities of the oriented case make the proof of the RSW result very different from the non-oriented case, and that the denomination simply refer to the fact that crossings of rectangles in the hard direction are expressed in terms of crossing of rectangles in the easy direction.

\paragraph{Generalization to other two-dimensional models} We work with a specific choice of model but we believe that the proof extends {\em mutatis mutandis} to oriented percolation on $\mathbb Z^2$ where edges are oriented from $x$ to $x+(0,1)$, $x+(-1,0)$ and $x+(0,1)$, and to the geometric representation of the one-dimensional contact process.

\paragraph{Applications and open problems} For non-oriented percolation, non-trivial bounds on crossing probabilities is the key step towards the understanding of the critical and near-critical phases. We believe that the box-crossing property established in this paper should lead to similar applications in the oriented case. For instance, scaling relations can be studied using \cite{DurTan89,DurSchTan89b}, see \cite{preparation}.

Let us mention that studying the limit of $\sfC_0$ conditioned on $0\rightarrow\ell_n$ and computing the exact value of critical exponents is a major open question. In particular, two objects of special interest in the oriented case are the set of ``renewal points'' (i.e.~heights that intersect $\sfC_0$ only once), and the process of the ``right-most particle'' $n\mapsto R_n$, see Fig.~\ref{fig:simulation}.

     \begin{figure}
    \label{fig:simulation}
    \centering
    \begin{minipage}[c]{.15\linewidth}
      \includegraphics[height=20cm]{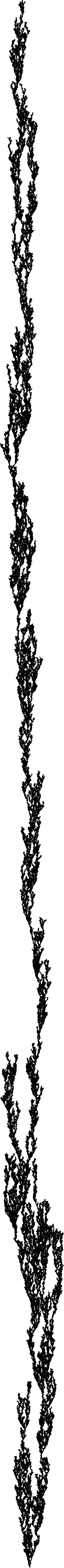}
    \end{minipage}
    \begin{minipage}[c]{.8\linewidth}
      \includegraphics[width=\textwidth]{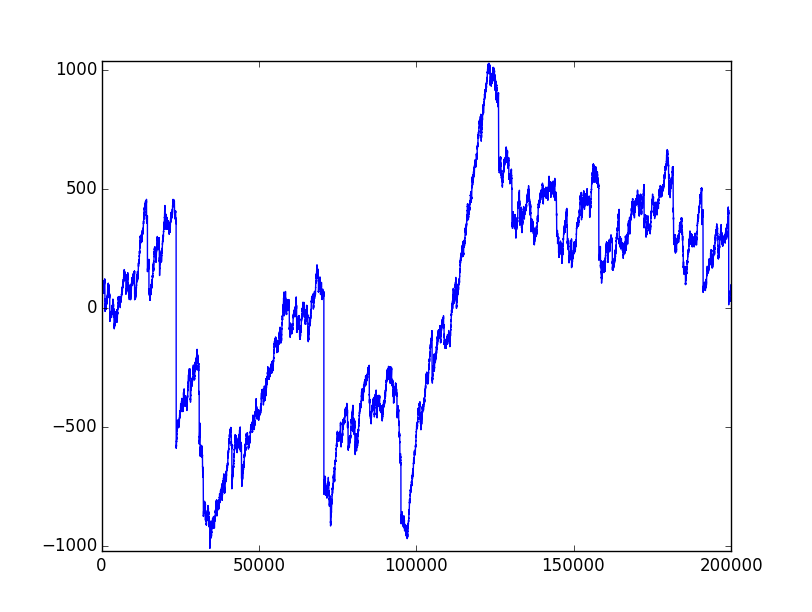}
      \caption{On the left, $\sfC_0$ conditioned on $0\rightarrow \ell_{8000}$ and $0\not\rightarrow \ell_{10000}$. Above, the process obtained by taking the right-most particle of $\sfC_0$ conditioned on $0\rightarrow\ell_{200\,000}$.}
    \end{minipage}

  \end{figure}

\subsection{Preliminaries}
\label{sec:preliminaries}

\paragraph{Further notation} We will always work with intersections of sets with $\mathbb L$. For instance, $[a,b]\times[c,d]$ will mean the intersection of $\mathbb L$ with the corresponding part of the plane. We write $A\rightarrow B$ for the event that there exist $x\in A$ and $y\in B$ with $x\rightarrow y$. Below, we will drop the subscript $p_c$ in the notation and write for instance $\bbP$, $H(m,n)$ and $V(m,n)$ for $\bbP_{p_c}$, $H_{p_c}(m,n)$ and $V_{p_c}(m,n)$. Importantly, we will keep the subscript $p$ when $p$ is not a priori equal to $p_c$.

\paragraph{One input from percolation theory: the square root trick} We will use repeatedly (see \cite{Gri99}) the classical Harris-Fortuin-Kasteleyn-Ginibre (FKG) inequality: for two increasing events\footnote{An event $E$ is increasing if it is stable by opening edges.} $E$ and $F$,
\begin{equation}
\label{eq:FKG}\tag{FKG} \bbP_p(E\cap F)\ge \bbP_p(E)\bbP_p(F).
\end{equation}
Let us also mention the following trivial application of the FKG inequality, called the {\em square-root trick}: for any increasing events $A_1,\dots,A_N$,
\begin{equation}\tag{SRT}\max\{\bbP_p(A_n):1\le n\le N\}\ge 1-\big(1-\bbP_p(A_1\cup\dots\cup A_N)\big)^{1/N}.\end{equation}

\paragraph{Organization of the paper} Section 2 is devoted to the proof of the Russo-Seymour-Welsh type result. This result is then used in Section 3 to derive the box-crossing property. Section 4 is devoted to the proofs of Theorems~\ref{thm:1} and \ref{thm:2}.

\paragraph{Acknowledgments} The work of the two first authors was supported by a grant from the Swiss NSF and the NCCR SwissMap also funded by the Swiss NSF. The project was initiated during a stay of the third author to the Universit\'e de Gen\`eve, and the authors are grateful to the institution for making such a stay possible. AT was supported by CNPq grants 306348/2012-8 and 478577/2012-5 and by FAPERJ grant 202.231/2015.

\section{Russo-Seymour-Welsh type result}\label{s:rsw}

This section is dedicated to the proof of a Russo-Seymour-Welsh theorem for oriented percolation. It enables us to express crossing probabilities of rectangles with different aspect ratios. We include also a technical (and easy) result at the end of this section. In this section it will be convenient to use crossing probabilities for rectangle which may have non integer dimensions. If $r,s$ are two real numbers, we set
\begin{equation}
  \label{eq:21}
  H_p(r,s)=H_p(\lceil r\rceil,\lceil s\rceil)\quad\text{and}\quad V_p(r,s)=V_p(\lceil r\rceil,\lceil s\rceil),
\end{equation}
where $\lceil r\rceil$ denotes the upper integer part of~$r$.
\ncg{g:RSW}
\ncg{g:width}
\begin{theorem}[RSW type result]
  \label{t:RSW}
For any $\alpha\in(\tfrac34,1)$, there exist $\varepsilon\in(0,1)$ and an increasing homeomorphism $\ucg{g:RSW}: [0, 1] \to [0,1]$ such that for any $m,n\ge1$,
  \begin{equation*}
    \min\big\{V_p(m, 3n),H_p(3m,n)\big\} \geq \ucg{g:RSW}\big( \min \big\{ V_p (m, \alpha \varepsilon n),H_p(\alpha m,\varepsilon n) \big\} \big).
  \end{equation*}
\end{theorem}

On the left, if a rectangle of size $m$ times $n$ is our reference, the crossing probabilities involve rectangles which are three times longer in the direction of crossing. On the right, if a rectangle of size $m$ times $\varepsilon n$ is our reference, the crossing probabilities involve rectangles which are slightly shorter in the direction of crossing. This is reminiscent of the classical RSW theory for non-oriented percolation: crossing probabilities in the hard direction can be bounded from below by expressions involving crossing probabilities in the easy direction.

The heights of the rectangles are very different on the left and the right (there is a factor roughly $\varepsilon$ between the two), which is a major difference between the oriented and the non-oriented cases. Said differently, in order to obtain estimations on probabilities of crossings of rectangles in the hard direction, one needs to pay a cost on the height of the rectangle. The following example illustrates perfectly why changing the height is necessary in the oriented case: think of the extremal case of the horizontal crossing from left to right of a rectangle of size $n$ times $n$. In this case, it is simply impossible to cross horizontally a rectangle of size $2n$ times $n$ due to the direction of the edges.

We start the proof of the theorem by a key lemma allowing us to increase the width of rectangles which are crossed horizontally.

\begin{lemma}\label{lem:hor}
For any $\alpha\in(\tfrac34,1)$, there exists an increasing homeomorphism $\ucg{g:width}: [0, 1] \to [0,1]$ such that for any $k,\ell\ge 1$,
  \begin{equation*}
    H_p(k,\ell) \geq \ucg{g:width} \big( \min \big\{ V_p (k,\ell),H_p(\alpha k,\ell/2) \big\} \big).
  \end{equation*}
\end{lemma}
Note that in the statement above, we allow the variables $k$ and $\ell$ to take non-integer values.

\begin{proof}
  Let us first assume that $k/2$, $\ell/2$ and $\alpha k$ are integers (this is purely for convenience as can be seen at the end of the proof). Introduce the boxes
\begin{equation*}B=[-k/2,k/2]\times[0,\ell]\quad\text{and}\quad
B_r=[0,\alpha k]\times[\ell/2,\ell].
\end{equation*}
illustrated on Fig.~\ref{fig:RSW1}. Let $E$ be the event that there exists an open path in $B\cup B_r$ starting from the bottom of $B$ and ending on the right side of $B_r$. Let us prove that
\begin{equation}\label{eq:aa}\mathbb P_p[E]\ge H_p(\alpha k,\ell/2)\Big(1-\sqrt{1-V_p(k,\ell)}\Big).\end{equation}
In order to get this inequality, we use a ``conditioning on the top-most left-right crossing of $B_r$'' illustrated on Fig.~\ref{fig:RSW1}. This type of reasoning is now classical in percolation.
\begin{figure}[htpp]
  \label{fig:RSW1}
  \centering
  \includegraphics[width=0.34\textwidth]{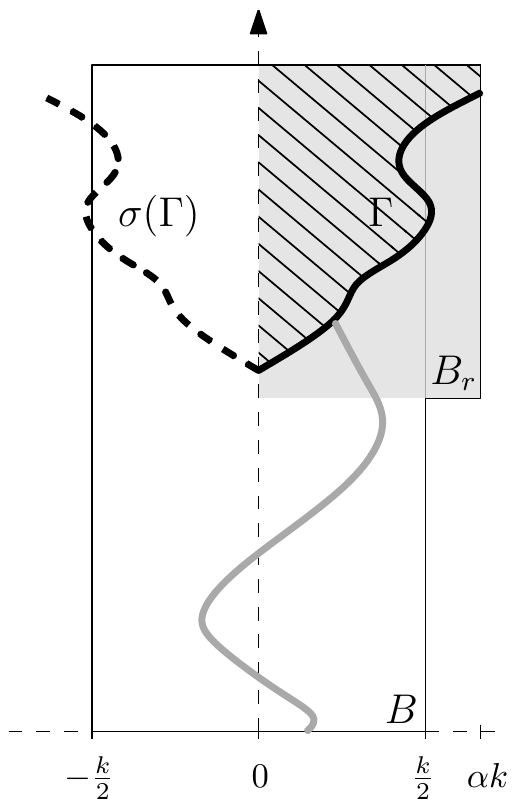}
  \caption{Construction of the event $E$. First, we require that the box $B_r$ be crossed from left to right, and we explore the top-most left-right crossing $\Gamma$ in $B_r$. After this exploration, the edges in the hatched region have been discovered. Then we ask that in the unexplored region there exists an open path (in grey) connecting the bottom side of $B$ to $\Gamma$.}
\end{figure}

For a configuration $\omega$ containing a left-right crossing of $B_r$, define $\Gamma$ to be the top-most left-right crossing\footnote{Formally, this can be seen as the largest left-right crossing of $B_r$ for the natural lexicographical order on path induced by the lexicographical order on vertices and the order that the edge going left from a vertex is smaller than the edge going right.} of $B_r$. When there is no such crossing, set $\Gamma=\emptyset$. Since the box $B_r$ is crossed from left to right with probability $H(\alpha k, \ell/2)$, we have
 \begin{equation}
   \label{eq:20}
   H(\alpha k,\ell/2)=\sum_{\gamma\neq \emptyset}\mathbb P_p ( \Gamma=\gamma),
 \end{equation}
 where the sum is over all the possible left to right paths in $B_r$. Fix for a moment such a path $\gamma$. Introduce the orthogonal symmetry $\sigma$ with respect to the axis $y=0$. Define $S_\gamma$ to be the set of vertices of $B$ which are reachable from a vertex of the bottom of $B$ by an oriented path of edges not crossing $\gamma\cup\sigma(\gamma)$. Let $E_\gamma$ be the event that there exists a path in $S_\gamma$ connecting the bottom side of $B$ to $\gamma$ inside $S_\gamma$. Using symmetry and the square root trick, together with the fact that any path crossing $B$ vertically must contain a path reaching $\gamma$ or $\sigma(\gamma)$ in $S_\gamma$, we find
\begin{equation}
  \label{eq:29}
  \mathbb P_p(E_\gamma)\ge 1-\sqrt{1-V_p(k,\ell)}.
\end{equation}
Now, if $\Gamma=\gamma$ and $E_\gamma$ occurs then the event $E$ occurs. Therefore, summing over all the possible paths $\gamma$, we obtain
\begin{align}
    \mathbb P_p(E)&\ge \sum_{\gamma\neq\emptyset}\mathbb P_p(\{\Gamma=\gamma\}\cap E_\gamma)\notag\\
    &= \sum_{\gamma\neq\emptyset}\mathbb P_p(\Gamma=\gamma) \mathbb P_p( E_\gamma).\label{eq:32}
\end{align}
In the second line, we used that the event $\Gamma=\gamma$ is measurable with respect to the edges with both ends in $B\setminus S_\gamma$ while $E_\gamma$ is measurable with respect to the edges in $S_\gamma$, therefore these two events are independent. We finally obtain Eq.\eqref{eq:aa} by combining the equation above together with \eqref{eq:20} and \eqref{eq:29}.

We now conclude the proof. Consider the boxes
\begin{align*}
  B_\ell&=[k/2-\alpha k,k/2]\times[0,\ell/2],\\
B'&=[k/2-\alpha k,\alpha k]\times[0,\ell].
\end{align*}

\begin{figure}
  \label{fig:RSW2}
  \begin{minipage}[t]{.48\linewidth}
    \centering
    \includegraphics[width=0.7\textwidth]{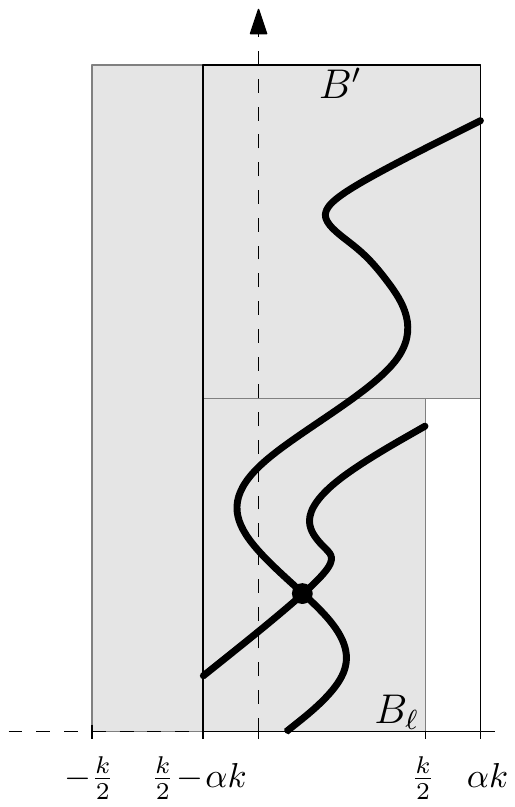}
  \end{minipage}
  \hfill
  \begin{minipage}[t]{.48\linewidth}
    \centering
    \includegraphics[width=0.7\textwidth]{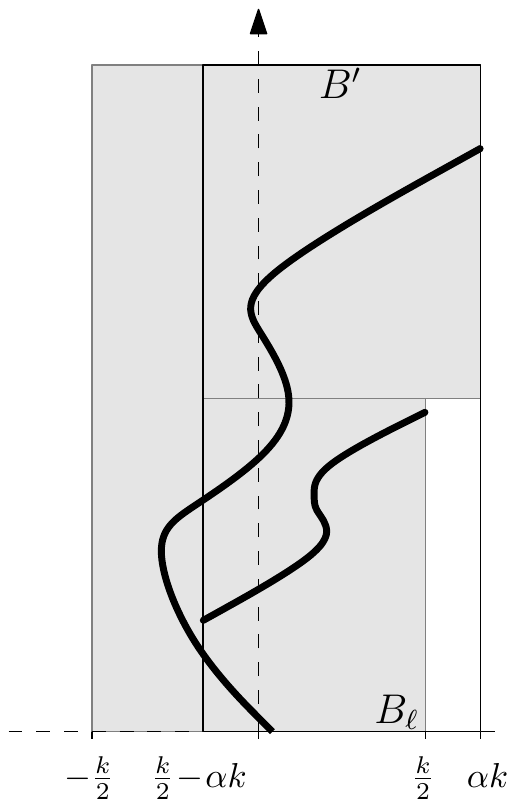}
  \end{minipage}
\caption{Two possible cases when  the event $E\cap C_\ell$ occurs: on the left  picture, the left-right crossing of $B_\ell$ intersect the path realizing $E$, and on the the right picture the two paths do not intersect. In both cases, the box $B'$ is crossed from left to right.}
\end{figure}
Let $C_\ell$ be the event that $B_\ell$ is crossed from left to right. On the event $E\cap C_\ell$, there must exist a path from left to right in the box $B'$. Indeed, we are in one of the two following cases (illustrated on Fig.~\ref{fig:RSW2}):
\begin{itemize}
\item A crossing from left to right in $B_\ell$ intersects a crossing from the bottom of $B$ to the right of $B_r$, thus creating a left-right crossing in $B'$.
\item No crossing from left to right in $B_\ell$ intersects a crossing from the bottom of $B$ to the right of $B_r$, in such case any of the latter paths contains a left-right crossing $B'$.
\end{itemize}
Since $2\alpha-\tfrac12>1$, we deduce that
\begin{align*}H_p(k,\ell)&\stackrel{\phantom{\rm (FKG)}}\geq \mathbb P_p(B'\text{ is crossed from  left to right})\\
  &\stackrel{\phantom{\rm (FKG)}}\ge \mathbb P_p(E\cap C_\ell)\\
  &\stackrel{\rm (FKG)}{\ge} \mathbb P_p(C_\ell)\mathbb P_p(E)\\
  &\stackrel{\ \eqref{eq:aa}\ }\ge H_p(\alpha k,\ell/2)^2\Big(1-\sqrt{1-V_p(k,\ell)}\Big).\end{align*} This finishes the proof of the case where $k$ and $\ell$ are two even integers. For general real values $k$ large enough and $\ell\ge1$, one may do the same proof with $B=[-\lceil k/2\rceil , \lceil k/2\rceil]\times[0,\lceil \ell \rceil]$ and $B_r=[0,\lceil k/2\rceil ]\times [\lceil \ell\rceil - \lceil \ell/2\rceil,\lceil \ell \rceil]$ provided that $2\lceil \alpha k\rceil-\lceil k/2\rceil \ge k$. Finally, note that by choosing $\ucg{g:width}$ properly, we may cover the case of small values of $k$.
\end{proof}

The next trivial lemma will be useful in the proof. For $k,\ell\ge1$ integers, let $E(k,\ell)$ be the event that $\{0\}\times[0,\ell]$ is connected to $\{k\}\times[2\ell,3\ell]$ or $[0,k]\times\{3\ell\}$ (see Fig.~\ref{E(k,ell)}).
  \begin{lemma}\label{lem:vert}
    For any integer $C>0$ and any integers $k,\ell\ge1$,
  $$V_p(k,C\ell)\ge \mathbb P_p\big(E(k,\ell)\big)^{2C}.$$
  \end{lemma}

  \begin{proof} For an integer $i\ge0$, let $F_i$ be the event that $\{0\}\times[i \ell,(i + 1) \ell]$ is connected to $\{0\} \times [(i + 2) \ell, (i + 3) \ell]$ inside the strip $[0,k]\times\mathbb Z$. First, by translation invariance, the probability of $F_i$ is equal to the probability of $F_0$. Then, observe that the event $F_0$ occurs as soon $E(k,\ell)$ occurs together with a symmetric version of it (see Fig.~\ref{F0}). Therefore, by the FKG inequality, we have for every $i\ge 0$
    \begin{equation*}
      \mathbb P_p(F_i)=\mathbb P_p(F_0)\ge \mathbb P_p(E(k,\ell))^2.
    \end{equation*}
Finally, if all the events  $F_i$ occur for $0\le i< C$, the box $[0,k]\times[\ell,\ell+C\ell]$ is crossed vertically. The lemma thus follows from the FKG inequality.
\end{proof}

\begin{figure}
  \begin{minipage}[t]{.47\linewidth}
    \label{E(k,ell)}
    \centering
    \includegraphics[width=0.6\textwidth]{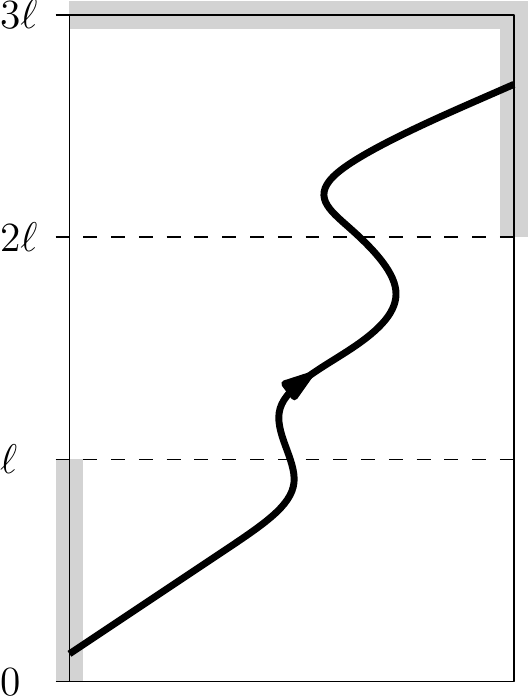}
    \caption{Diagrammatic representation of the event $E(k,\ell)$.} \end{minipage} \hfill
  \begin{minipage}[t]{.47\linewidth}
    \label{F0}
    \centering
    \includegraphics[width=0.6\textwidth]{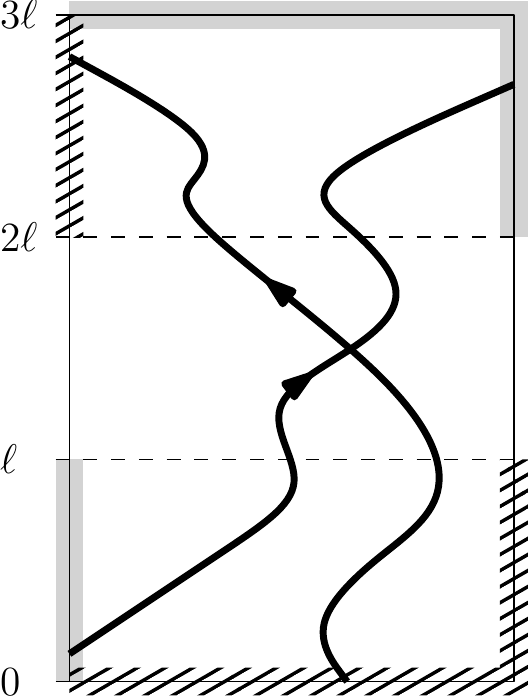}
    \caption{The event $F_0$ obtained by intersecting $E(k,\ell)$ and a symmetric version of it.}
  \end{minipage}
  \end{figure}

  \begin{proof}[Proof of Theorem~\ref{t:RSW}] Without loss of generality, we may assume that $\alpha m$ is an integer. We start by proving the bound on $H_p(3m,n)$ assuming the bound on $V_p(m,3n)$. For $k\ge m$ and $\ell\le 3n$, Lemma~\ref{lem:hor} implies
$$H_p(k,\ell)\ge \ucg{g:width}\big(\min\big\{V_p(m,3n),H_p(\alpha k,\ell/2)\big\}\big).$$
By iterating the statement above $s$ times, we get for every $s\ge 1$
\begin{equation*}
  H_p(\alpha^{1-s}m,n)\ge \ucg{g:width}^{(s)}\big(\min\big\{V_p(m,3n),H_p(\alpha m,n/2^s)\big\}\big).
\end{equation*}
Fix $s = s(\alpha)$ such $\alpha^{1-s}\ge 3$ and set $\varepsilon = \varepsilon(\alpha)=2^{-s}$. Then the equation above implies the desired inequality. Note that this is the only place where the constant $\varepsilon$ is used: it guarantees that the height of the rectangles obtained via the iteration of Lemma~\ref{lem:hor} is always smaller than $n$ (and hence a fortiori $3n$).
\ncg{2}
\ncg{4} \ncg{5} \ncg{6} \bigbreak Let us now focus on the lower bound on $V_p(m,3n)$. Let $\ell= \alpha \varepsilon n/12$ and let $\ucg{2}$ be an homeomorphism defined through:
\begin{equation}
  g_*(x) = 1 - (1 - x)^{1/12}, g_{\#}(x) = 1 - (1 - x)^{1/2} \text{ and } \ucg{2}(x) = g_{\#} \circ g_* (x).
\end{equation}

We may assume without loss of generality that $\ell$ is an integer. We divide the proof in two cases.

\noindent {\bf Case 1.} $H_p(\alpha m,2\ell)< \ucg{2}(H_p(\alpha m,\varepsilon n))$.
\bigbreak
For $i=0,\dots,11$, let $A_i$ be the event that there exists an open path from $\{0\}\times[i\ell,(i+1)\ell]$ to $\{\alpha m\}\times[0,12\ell]$ in the strip $[0,\alpha m]\times\mathbb Z$. Since for every $i$, $\mathbb P_p(A_0)\ge\mathbb P_p(A_i)$, the square-root trick implies that there exists some $i$ with
$$\mathbb P_p(A_0)\ge 1-(1-H_p(\alpha m,\varepsilon n))^{1/12} = g_*\big(H_p(\alpha m, \varepsilon n)\big).$$
Now, if $A_0$ occurs, then either $[0,\alpha m]\times[0,2 \ell]$ is crossed horizontally, or the event $E(\alpha m,\ell)$ occurs. As a consequence, the square-root trick used one more time implies that
$$\max\{H_p(\alpha m,2\ell),\mathbb P_p(E(\alpha m,\ell))\}\ge \ucg{2}(H_p(\alpha m,\varepsilon n)).$$
(This is the definition of $\ucg{2}$ used above.)
The assumption on $H_p(\alpha m,\varepsilon n)$ implies that
$$\mathbb P_p(E(\alpha m,\ell))\ge \ucg{2}(H_p(\alpha m,\varepsilon n)),$$ so that Lemma~\ref{lem:vert} applied to $k=\alpha m$, $\ell$ and $C>16/\alpha\varepsilon$ gives
$$V_p(m,3n)\ge \ucg{4}(H_p(\alpha m,\varepsilon n)).$$

\noindent{\bf Case 2.} $H_p(\alpha m,2\ell)\ge \ucg{2}(H_p(\alpha m,\varepsilon n))$.
\bigbreak
In such case, Lemma~\ref{lem:hor} implies that
\begin{align}
  \label{eq:bc}H_p(m,4\ell) & \ge\ucg{g:width}( \min \big\{ V_p (m, 4 \ell), \ucg{2}(H_p(\alpha m, 2 \ell))\big\})\\
  \nonumber
  &\ge\ucg{g:width}( \min \big\{ V_p (m,\alpha\varepsilon n), \ucg{2}(H_p(\alpha m,\varepsilon n))\big\}).
\end{align}
Since $E(m,4\ell)$ occurs as soon as there exists a left-right crossing of $[0,m]\times[0,4\ell]$ and a vertical crossing of $[0,m]\times[0,12\ell]$, the FKG inequality implies immediately that
\begin{equation}\label{eq:bd}\mathbb P_p(E(m,4\ell))\ge H_p(m,4\ell)V_p(m,12\ell).\end{equation}
Since $12\ell\le \alpha \varepsilon n$, \eqref{eq:bc} and \eqref{eq:bd} can be combined to obtain
$$\mathbb P_p(E(m,4\ell))\ge \ucg{5}(\min\{V_p (m,\alpha\varepsilon n),H_p(\alpha m,\varepsilon n)\big\}).$$
Lemma~\ref{lem:vert} applied with $k=m$, $\ell$ and $C>8/\alpha\varepsilon$ gives
\begin{equation}
V_p(m,3n)\ge \ucg{6}(\min\{V_p (m,\alpha\varepsilon n),H_p(\alpha m,\varepsilon n)\big\}),
\end{equation}
thus concluding the proof in this case as well.
\end{proof}

Let us mention the following technical statement, which will be useful in the next sections.
\ncg{easy}
\begin{lemma}\label{lem:from 2 to alpha}
For any $\Delta>\delta>1$, there exists $C>0$ such that for any $n,m\ge1$,
  \begin{equation*}
   \max\{V_p(\Delta m, n),H_p(m,\Delta n)\}\le \ucg{easy}\big(\max\big\{V_p (\delta m, n),H_p(m,\delta n) \big\}\big),
  \end{equation*}
  where $\ucg{easy}(x)=1-(1-x)^C$ for any $x\in[0,1]$.
\end{lemma}

\begin{proof}
  Let us present the proof for $V_p(\Delta m, n)$ (the proof for $H_p(m,\Delta n)$ can be adapted easily). Set $\varepsilon< (\delta-1)$ and an integer $K>\Delta/\varepsilon$. We may assume without loss of generality that $\varepsilon n$ and $\delta n$ are two integers.

Define the two collections of boxes
  \begin{align*}
    \mathcal{F} &= \big\{ [k\varepsilon m,(k\varepsilon +\delta)m]\times[0, n],0\le k< K \big\},\\
    \mathcal{E} &= \big\{ [k\varepsilon m,(k\varepsilon+1)m]\times[0,n], 0\le k<  K \big\}.
  \end{align*}
For $[0,\Delta m]\times[0,n]$ to be crossed vertically, then one of the boxes in $\mathcal F$ must be crossed vertically, or one of the boxes in $\mathcal E$ must be crossed from left to right, or one of the boxes in $\mathcal E$ must be crossed from right to left. In other words, the event that $[0,\Delta m]\times[0,n]$ is crossed vertically is contained in the union of $3K$ events of probability smaller or equal to $V_p(\delta m, n)$ and $H_p(m,n)(\le H_p(m,\delta n))$. The square-root trick implies that
$$V_p(\Delta m,n)\le   1-(1-x)^{3K},$$
where $x:=\max\big\{V_p (\delta m, n),H_p(m,\delta n) \big\}$. The proof follows by setting $C=3K$.
\end{proof}

\begin{remark}\label{rmk:correlation}
Combined with Theorem~\ref{t:box_crossing} below, Lemma~\ref{lem:from 2 to alpha} shows that for any $\Delta>1$, $H(w_n,\Delta n)$ and $V(\Delta w_n,n)$ are bounded by $1-c(\Delta)<1$ uniformly in $n\ge1$.
\end{remark}

\section{The box-crossing property}

This section is devoted to the proof of Theorem~\ref{t:box_crossing}.
With the help of the RSW result from the previous section, the proof of the theorem is not more than a proper definition for $w_n$. Theorem~\ref{t:RSW} does the work for us, since it enables us to invoke two classical results on crossing probabilities (see the lemma below), which are somehow not specific to oriented percolation.
\nc{c:exp_decay}
\begin{lemma}[finite size criteria for $p<p_c$ and $p>p_c$]
  \label{t:Kesten}
 There exists $\eta > 0$ such that for $p \in (0, 1)$ and $m, n \geq 1$,
 \begin{itemize}
 \item If $\max\{V_p(2m, n),H_p(m,2n)\} < \eta$,
  then $p < p_c$ and there exists $c>0$ such that for any $N\ge1$,
  \begin{equation*}
    \mathbb{P}_p \big( 0 \lr \ell_N \big) \leq \exp(-c N).  \end{equation*}
\item If $\min\{V_p (m, 2n) ,H_p (2m, n)\} > 1 - \eta$,
  then $p>p_c$ and there exists $c>0$ such that for any $N\ge1$,
  \begin{equation}
    \label{e:exp_decay_super_2}
    \mathbb{P}_p \big( 0 \lr \ell_N, 0 \not \lr \infty \big) \leq \exp (-c N).
  \end{equation}
  \end{itemize}
\end{lemma}

Before proving this lemma, let us show the theorem. Recall that we omit the subscript $p_c$.
\begin{proof}[Proof of Theorem~\ref{t:box_crossing}]
Let $n\ge1$ large enough. Set $\eta$ to be the constant in the previous lemma.  Fix any $\alpha \in (3/4, 1)$ and $\varepsilon = \varepsilon(\alpha)>0$ as in Theorem~\ref{t:RSW}.
Introduce
\begin{equation}\label{eq:abca}
w_n:=\inf\big\{m\ge 0: H(\alpha m,\varepsilon n)\le V(m,\alpha \varepsilon n)\big\}.
\end{equation}
Note that $w_n$ diverges with $n$.
Introduce the following notation:
\begin{align*}
H_-&:=H(\alpha (w_n-1),\varepsilon n) &H_+:=H(\alpha w_n,\varepsilon n)\\
V_-&:=V(w_n-1,\alpha\varepsilon n)&V_+:=V(w_n,\alpha \varepsilon n).
\end{align*}
The definition of $w_n$ implies that $H_+\le V_+$ and $H_->V_-$.
\medbreak
\noindent{\bf Proof of the lower bound.} Choosing $\delta \in (1, 1/\alpha)$, Lemma~\ref{lem:from 2 to alpha} implies that
\begin{equation*}
  \begin{split}
    \max\{V(2w_n,n),H(w_n,2n)\}& \leq \ucg{easy}\big( V(\delta\alpha w_n, \alpha \varepsilon n) \vee H(\alpha w_n, \delta \alpha \varepsilon n) \big)\\
    & \leq \ucg{easy}(\max\{V_-,H_+\}).
  \end{split}
\end{equation*}
The first item of Lemma~\ref{t:Kesten} thus implies that $\max\{V_-,H_+\}\ge \ucg{easy}^{-1}(\eta)$. This gives that either $H_->V_-\ge \ucg{easy}^{-1}(1 - \eta)$, or $V_+\ge H_+\ge \ucg{easy}^{-1}(\eta)$. In either case, Theorem~\ref{t:RSW} may be applied to get
$$\min\{V(w_n,3n),H(3w_n,n)\}\ge \ucg{g:RSW}(\ucg{easy}^{-1}(\eta)).$$
\medbreak
\noindent{\bf Proof of the upper bound.} Theorem~\ref{t:RSW} implies that
\begin{equation*}\min\{V(3w_n,n),H(w_n,3n)\}\ge \ucg{g:RSW}(\min\{V_+,H_-\}).\end{equation*}
The second item of Lemma~\ref{t:Kesten} thus implies that $\min\{V_+,H_-\}\le \ucg{g:RSW}^{-1}(1-\eta)$. This gives that either $H_+\le V_+\le\ucg{g:RSW}^{-1}(\eta)$, or $V_- <H_-\le \ucg{g:RSW}^{-1}(1-\eta)$. In either case, Lemma~\ref{lem:from 2 to alpha} may be applied to get
$$\max\{V(3w_n,n),H(w_n,3n)\}\le \ucg{easy}(\ucg{g:RSW}^{-1}(1-\eta)).$$
\end{proof}

\begin{proof}[Proof of Lemma~\ref{t:Kesten}]{\bf Proof of the first item.} Introduce the sequence of scales
$m_k = 2^k m$ and $ n_k = 2^k n$
  for $k \geq 0$ and set
  \begin{equation*}
    u_k = \max \big\{ H_p (m_k, 2 n_k), V_p(2 m_k, n_k) \big\}.
  \end{equation*}
A vertical crossing of the box $[0,4m_{k}]\times[0,2n_{k}]$ must contain vertical crossings of the boxes $[0,4m_k]\times[0,n_k]$ and $[0,4m_k]\times[n_k,2n_k]$. Lemma~\ref{lem:from 2 to alpha} (and the trivial bound $\ucg{easy}(x):=1-(1-x)^C\le Cx$) thus implies that
$V_p(2m_{k+1},n_{k+1})\le (Cu_k)^2.$
Doing the same with $H_p(m_{k+1},2n_{k+1})$, we deduce that
  \begin{equation*}
    u_{k + 1} \leq (Cu_k)^2.
  \end{equation*}
By choosing $\eta<\tfrac1{eC^2}$ small enough, $u_0<\eta$ implies that
$u_k\le \exp(-2^k)$ for any $k\ge0$.
To conclude, fix $N\geq 1$ and let $K$ be the unique integer such that $n_K\le N<2n_K$. The event $0\rightarrow\ell_N$ implies that one of the three rectangles  $[-m_{K}, m_{K}] \times [0, n_{K}]$, $[0,m_{K}]\times[2 n_{K}]$ or $[-m_{K}, 0] \times [0, 2 n_{K}]$ must be crossed ``in the easy direction'', we deduce that
  \begin{equation*}
    \mathbb{P}_p \big( 0 \lr \ell_N\big) \leq 3 u_{K}\leq \exp \{ -c N\},
  \end{equation*}
for a constant $c>0$ small enough. This finishes the proof of exponential decay.

The fact that $p<p_c$ follows from
the observation that the condition $\max\{V_p(2m, n),H_p(m,2n)\}< \eta$ is satisfied for some $p'>p$, and that therefore $p<p'\le p_c$.
\bigbreak
\noindent{\bf Proof of the second item.}  For this proof, we consider a dependent percolation defined on a renormalized lattice.
  More precisely, given integers $m, n \geq 1$, associate to every $x = (i, j) \in \mathbb{L}$ the boxes
  \begin{align*}
      B_x &:= [0,m]\times[0,2n]~+~ (im, jn),\\
      B_x^+ &:= [0,2m]\times[0,n] ~+~ \big( im, (j + 1)n \big),\\
  B_x^- &:= [0,2m]\times[0,n]~+~ \big( (i - 1)m, (j + 1)n \big).
    \end{align*}

  Say that the edge $(x, x + (1, 1))$ is {\em open} if $B_x$ is crossed vertically and $B_x^+$ is crossed from left to right.
  Analogously, say that the edge $(x, x + (-1, 1))$ is {\em open} if $B_x$ is crossed vertically and $B_x^-$ is crossed from right to left.   Denote the induced percolation measure $\mathbb{P}^{m, n}_p$.

  On the event that there is an infinite path of open edges starting from the origin (using the above definition), then there is also an infinite open path on the original lattice, starting from $[0,m]\times[0,n]$.

  Note that the above percolation measure is $3$-dependent, as defined below (7.60) of \cite{Gri99}.
  Therefore, using a result by Liggett, Schonmann and Stacey (see Theorem~(7.65) of \cite{Gri99}), we conclude that there exists an $\varepsilon > 0$ such that
  if
  \begin{equation}
    \label{eq:142}
    \mathbb{P}^{n, m}_p \Big( \big( (0, 0), (1, 1) \big) \text{ is open} \Big) > 1 - \varepsilon,
  \end{equation} then there exists $c>0$ such that for every $N\ge1$,
$$\min\{V_p(N, 2N), H_p(2N, N)\} \geq 1 - \exp(-c N),$$
(see for instance the contour counting argument presented in Section~10 of \cite{Dur84} for additional details). The claim follows since \eqref{eq:142} is directly implied by the assumption in the statement.

The above implies that $p > p_c$ since $\min\{V_p(m, 3n),H_p(3m,n)\}>1-\eta$ is satisfied for some $p'<p$.
\end{proof}

\section{Proofs of the main theorems}
\subsection{Relation between $R_n$ and $w_n$}
\label{sec:stand-devi-r_n}

In this section we use the box-crossing property to show that $w_n$ is equal up to constant to several quantities related to $R_n$. The two first items below will be useful to obtain polynomial bounds on $w_n$. The last two items are useful to get the main theorems. Below, $x^+=\max\{x,0\}$.

\begin{proposition}
  \label{prop:equiv}
  \nc{c:1}\nc{c:3}\nc{c:4}\nc{c:5}
  There exist constants $\uc{c:1},\uc{c:3},\uc{c:4},\uc{c:5}>0$ such that for every $n\ge1$,
  \begin{enumerate}[$(i)$]
  \item\label{item:1} $\uc{c:1} w_n \le \mathbb E(R_n^+)  \le \frac1{\uc{c:1}}w_n$,
  \item\label{item:2} $\uc{c:3} w_n \le \sqrt{\Var(R_n)} \le \frac1{\uc{c:3}}w_n$,
  \item\label{item:3} $\uc{c:4} w_n \le \mathbb E(R_n\:\vert\:0\lr\ell_n) \le \frac1{\uc{c:4}}w_n$,
  \item\label{item:4} $\uc{c:5} w_n \le \sqrt{\Var(R_n \:\vert\:0\lr\ell_n)} \le \frac1{\uc{c:5}}w_n$.
  \end{enumerate}
\end{proposition}

The proof of this proposition is heavily based on the box-crossing property. In particular, we will use several times the following event $E$, whose probability is bounded from below using the box-crossing property. Let $B=[-\tfrac12w_n,\tfrac52w_n]\times[0,n]$ and define $E$ (see Fig.~\ref{construction}) to be the event that there exist an open path from $[-\tfrac12w_n, 0]\times \{0\}$ to $[2w_n,\tfrac52w_n]\times\{n\}$ in $B$.
The event $E$ occurs if there exit
    \begin{itemize}
    \item a vertical crossing from $[-\tfrac12w_n, 0]\times \{0\}$ to the top side of $B$,
    \item a vertical crossing from the bottom side of $B$ to $[2w_n,\tfrac52w_n]\times\{n\}$,
    \item a left-right crossing of $B$.
    \end{itemize}
    By the box-crossing property and symmetry, the two first paths exist with probability larger than $\uc{c:box_crossing}/2$, and the third with probability larger than $\uc{c:box_crossing}$. The FKG inequality implies that
    \begin{equation}
      \label{eq:17}
    \bbP(E)\ge\tfrac14 \uc{c:box_crossing}^3.
    \end{equation}
      \begin{figure}[h]
    \label{construction}
    \centering
 \includegraphics[width=1.00\textwidth]{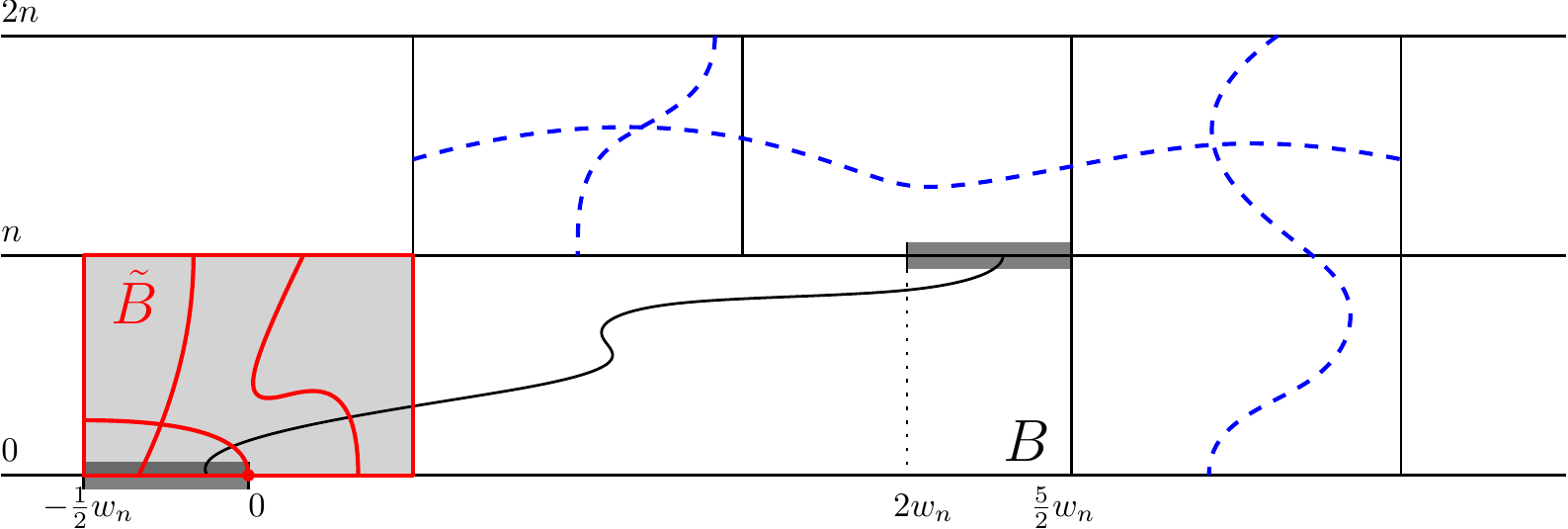}
    \caption{The crossing in black illustrates the occurrence of the event $E$. The red picture illustrates the bound $\bbP(0\rightarrow\ell_n)\ge \tfrac{\uc{c:box_crossing}^2}{4}\bbP(0\rightarrow \partial\tilde B)$: paths combine to obtain a path from 0 to $\ell_n$: if $0\rightarrow \partial \tilde B$ implies the existence of a path from $0$ to $\partial \tilde B\setminus \ell_n$, then the paths from $[-\tfrac12 w_n,0]\times\{0\}$ or $[0,\tfrac12 w_n]\times\{0\}$ would cross it to create a path from $0$ to $\ell_n$ in $\tilde B$. The dotted blue lines denoted dual path (which are not necessarily oriented) preventing the existence of oriented crossings.}
  \end{figure}

    We are now in a position to attack the proof of Proposition~\ref{prop:equiv}.
\begin{proof} We prove each item one after the other.
\bigbreak\noindent(\ref{item:1})
The lower bound is immediate since
    \begin{equation}
      \label{eq:18}
   \bbE(R_{n}^+)\ge 2w_n\bbP [R_{n}^+\ge 2w_n]\ge 2w_n\bbP(E)\stackrel{\eqref{eq:17}}\ge 2w_n\cdot \tfrac14 \uc{c:box_crossing}^3.
    \end{equation}
  The upper bound follows directly from the following exponential bound on the tail of $R_{n}$: for every $k\in \bbN$,
    \begin{equation}
      \label{eq:23}
      \bbP(R_{n} \ge kw_n)\le (1-\uc{c:box_crossing})^k,
    \end{equation}
which is obtained as follows. If $R_n\ge k w_n$, then there must exist an open path from left to right inside the rectangle $[0,kw_n]\times[0,n]$. In particular $k$ disjoint rectangles of size $w_n$ by $n$ must be crossed from left to right by an open path. This observation and independence imply that
    \begin{equation*}
      \label{eq:19}
      \bbP(R_{n}\ge k w_n) \le H(k w_n,n)\le H(w_n,n)^k\le H(w_n,3n)^k.
    \end{equation*}
The box-crossing property implies \eqref{eq:23}. By summing over $k$, we find that
    \begin{equation*}
      \bbE (R_n^+)\le w_n\sum_{k=0}^\infty\bbP(R_{n}\ge k w_n) \le \frac1{\uc{c:box_crossing}} w_n,
    \end{equation*}
    which gives the desired upper bound.
\bigbreak\noindent (\ref{item:2}) By \eqref{eq:18}, we already know that
$\bbP(R_{n}\ge 2w_{n})\ge \frac{\uc{c:box_crossing}^3}4$. Since $R_{n}\le w_n$ on the event that $[0,w_n]\times[0,n]$ is not crossed from left to right, we deduce that $\bbP(R_{n}\le w_n)\ge \uc{c:box_crossing}$. This directly implies that the lower bound on the standard deviation. 

    The upper bound follows once again from the following exponential bound on the tail of $\vert R_{n}\vert$: for every $k\in \bbN$,
    \begin{equation}
      \label{eq:23}
      \bbP(\vert R_{n} \vert \ge kw_n)\le (1-\uc{c:box_crossing})^k.
    \end{equation}
    The contribution of $R_n\ge0$ is controlled by \eqref{eq:19}. For the contribution of $R_n\le 0$, observe that  $R_n\le -kw_n$ implies that the rectangle $[-kw_n,0]\times[0,n]$ is not crossed vertically. In particular, $k$ disjoint rectangles of size $w_n$ by $n$  fail to be crossed vertically. Using independence, the box crossing property implies
    \begin{equation*}
      \label{eq:24}
      \bbP(R_n\le -kw_n)\le 1-V(kw_n,n)\le (1-V(w_n,n))^k\le (1-\uc{c:box_crossing})^k.
    \end{equation*}
\bigbreak
\noindent(\ref{item:3}) We use a technique similar to the proof of (\ref{item:1}). The lower bound is slightly more delicate here because we do not take the positive part of $R_n$ and we therefore have to show that the negative part does not counterbalance the positive part. To achieve this, we use that the law of the cluster of $0$ is invariant by the orthogonal reflection $\sigma$ with respect to the vertical axis $y=0$.

Let $F$ be the intersection of the event $E$ and its image by $\sigma$.
The FKG inequality together with \eqref{eq:17} implies that the event $F$ occurs with probability larger than $\uc{c:box_crossing}^6/16$.

Now, if $0$ is connected to $\ell_n$ and $F$ occurs, then $R_n$ must be larger than $2w_n$. Therefore,
\begin{align*}
    \bbE(R_n\mathds 1_{F\cap \{0\lr\ell_n\}})&\ge 2w_n\bbP(F \cap \{0\lr\ell_n\})\stackrel{(FKG)}\ge \tfrac{\uc{c:box_crossing}^6}{8}w_n\bbP(0\lr\ell_n).\label{eq:25}
\end{align*}
Furthermore, by invariance of $F$ under symmetry,
 \begin{align*}
  \bbE(R_n\mathds 1_{F^c\cap\{0\lr\ell_n\}})&=\tfrac12 \bbE(R_n\mathds 1_{F^c\cap\{0\lr\ell_n\}}) - \tfrac12 \bbE(L_n\mathds 1_{F^c\cap\{0\lr\ell_n\}})\ge 0,
\end{align*}
where $L_n$ is the left-most point of $\ell_n$ connected to $0$.

Summing the two displayed equations above and dividing by $\bbP(0\lr\ell_n)$ gives
\begin{equation}
  \bbE(R_n \:\vert\:0\lr\ell_n)\ge\tfrac{\uc{c:box_crossing}^6}{8}w_n.
\end{equation}
For the upper bound, we use an exponential domination as in \eqref{eq:24}. The only difference is that here we have to take care of the conditioning. Let $k\ge1$. If $(R_n\ge k w_n, 0 \rightarrow \ell_n)$, then there must exist an open path from $0$ to the boundary $\partial B$ of the box $B=[-w_n,w_n]\times[0,n]$, and a left-right crossing of $[w_n,kw_n]\times [0,n]$. Using independence and the box-crossing property, we obtain
\begin{align}
  \nonumber
  \bbP(R_n\ge kw_n, 0 \rightarrow \ell_n)& \le  H(w_n,n)^{k-1}\bbP(0\lr \partial B)\\
  & \le  (1-\uc{c:box_crossing})^{k-1}\bbP(0\lr \partial B).\label{eq:27}
\end{align}
To conclude the proof, we need to compare the probability of an open path from $0$ to $\partial B$ with the probability of an open path from $0$ to $\ell_n$. We use the following observation.  If $0$ is connected to $\partial B$ and the two rectangles $[-w_n,0]\times[0,n]$ and $[0,w_n]\times[0,n]$ are crossed vertically by open paths, then $0$ is connected to  $\ell_n$. The FKG inequality and the box crossing property imply
\begin{align}
   \bbP(0\lr \ell_n)\stackrel{\rm (FKG)}\ge V(w_n,n)^2\bbP(0\lr \partial B) \ge  \uc{c:box_crossing}^2\bbP(0\lr \partial B).\label{eq:28}
\end{align}
Plugging the inequality in \eqref{eq:27} and dividing by $\bbP(0\lr \partial B)$ gives
\begin{equation}
\bbP(R_n\ge kw_n\:\vert\: 0\lr \partial \ell_n)\le (1-\uc{c:box_crossing})^{k-1}/\uc{c:box_crossing}^2,\label{eq:31}
\end{equation}
which gives the claim after summing over $k$.
\bigbreak\noindent(\ref{item:4}) Lower bound - We already know from the previous part that $R_n\ge 2w_n$ with (conditional) probability larger than constant, so that we only need to prove that $\bbP(R_n\le \tfrac32 w_n|0\rightarrow \ell_n)\ge c$. In order to see that, let $\tilde B=[-\tfrac12w_n,\tfrac12w_n]\times[0,n]$ (see Fig.~\ref{construction}) and the event that
\begin{itemize}
\item[(i)] $[0,\tfrac12 w_n]\times\{0\}$ and $[-\tfrac12 w_n,0]\times\{0\}$ are both connected to $\ell_n$ by a path in $\tilde B$, and $0$ is connected to the boundary of $\tilde B$,
\item[(ii)] $[\tfrac12 w_n,\tfrac32 w_n]\times[0,n]$ is not crossed from left to right.
\end{itemize}
By symmetry and the box-crossing property, each of the two first paths are occurring with probability $\tfrac12\uc{c:box_crossing}$, therefore, the FKG inequality implies that the events in (i) occur with probability larger or equal to
$$\bbP((i)\text{ occurs})\ge\tfrac14\uc{c:box_crossing}^2\bbP[0\rightarrow \tilde B]\ge \tfrac14\uc{c:box_crossing}^2\bbP[0\rightarrow\ell_n].$$
Since the event in (ii) does not depend on edges in $\tilde B$, the box-crossing property implies
$$\bbP(R_n\le \tfrac32w_n|0\rightarrow \ell_n)\ge \frac{\bbP((i)\text{ occurs})\bbP((ii)\text{ occurs})}{\bbP[0\rightarrow\ell_n]}\ge \tfrac14\uc{c:box_crossing}^3.$$
%

The upper bound is a consequence of the following exponential domination. Recall that $L_n$ was defined in the proof of (\ref{item:2}) as the left-most point of $\ell_n$ connected to $0$ by an open path.  Using  $R_n\ge L_n$ and the fact that  $-L_n$ has the same law as $R_n$ (conditionally on the existence of an open path from $0$ to distance $n$), we obtain for every $k\ge1$
\begin{align}
  \label{eq:33}
  \bbP(|R_n|\ge kw_n\vert0\lr\ell_n )&\le \bbP(R_n\ge kw_n\vert0\lr\ell_n)+\bbP(-L_n\le -kw_n\vert0\lr\ell_n)\notag\\
  &\le 2\bbP(R_n\ge kw_n\vert0\lr\ell_n) \overset{\eqref{eq:31}}{\leq} 2(1-\uc{c:box_crossing})^{k-1}/\uc{c:box_crossing}^2.\notag
\end{align}
\end{proof}
\subsection{Polynomial bounds on $w_n$.}
\label{sec:lower-bound-typical}

We start by proving a polynomial lower bound on $w_n$ using the equivalence with $\sqrt{\Var(R_n)}$. 
\nc{c:11}
\begin{proposition}\label{cor:polybound}
  Fix $n\ge1$.   There exists a constant $\uc{c:11}>0$ such that for every $n\ge1$,
  \begin{equation}
    \label{eq:44}
    w_n\ge \uc{c:11}n^{2/5}.
  \end{equation}
\end{proposition}
\nc{c:19}
\begin{proof}
The starting point of the proof is given by (1.8) in \cite{DurSchTan89}, which shows that there exists a constant $\uc{c:19}>0$ such that
 \begin{equation}
    \label{eq:1}
    \Var(R_n)\ge \uc{c:19}\,n\, \mathbb P [ 0\lr\ell_n].
  \end{equation}
We refer the reader to the original paper for the argument. Let us simply say that it exploits a renewal structure of the right-most open path from $\mathbb  Z_-\times \{0\}$ to $\ell_n$ (this path ends at $(R_n,n)$) by showing that between two consecutive renewal heights, the horizontal increment of the path has variance larger than $\tfrac14$, and then showing that the expected number of renewal heights is at least $\uc{c:19}\,n\,\mathbb P [0\lr\ell_n]$.
\medbreak
For $x\in \{0,\ldots,w_n\}$, let $E(x)$ be the event that there exists a vertical crossing in $B(w_n,2n)$ that goes through the point $(x,n)$. Note that our choice of the lattice implies that the event $E(x)$ is empty when $x$ has a different parity from $n$. On the event $E(x)$, there exists an open path starting from $\ell_0$ and ending at $(x,n)$, and a path starting from $(x,n)$ and ending on $\ell_{2n}$. Hence we have, by independence and symmetry,
  \begin{equation*}
    \label{eq:45}
    \mathbb P(E(x))\le \mathbb  P [0\lr \ell_n]^2.
  \end{equation*}
  Furthermore, the box-crossing property and the union bound imply
  \begin{equation*}
    \label{eq:46}
    \uc{c:box_crossing}\le V(w_n,2n)\le \sum_{0\le x\le w_n}\mathbb P(E(x)).
  \end{equation*}
  The combination of the two equations above finally gives
  \begin{equation}
    \label{eq:47}
    \mathbb  P [0\lr \ell_n]\ge \frac{\uc{c:box_crossing}}{\sqrt{w_n}}.
  \end{equation}
Inserting  the bound on the variance of $R_n$ obtained via (\ref{item:2}) of Proposition~\ref{prop:equiv} in \eqref{eq:1} gives
  \begin{equation}
    \label{eq:48}
    \left(\frac{w_n}{\uc{c:3}}\right)^2 \ge \Var(R_n) \ge \uc{c:19}\,n\,  \mathbb  P [0\lr \ell_n]\ge  \uc{c:19}\,n\,  \frac{\uc{c:box_crossing}} {\sqrt{w_n}},
  \end{equation}
   which concludes the proof.
\end{proof}

We now show a polynomial upper bound on $w_n$ using the equivalence with $\mathbb E(R_n^+)$. The proof is based on sub-additivity properties of $\mathbb E(R_n^+)$ (see e.g.~\cite{Dur84} for background).

\begin{proposition}\label{lem:sublinear}
  There exists a constant $\epsilon>0$ such that for every $n\ge1$,
  \begin{equation}
    \label{eq:9}
    \mathbb E(R_n^+)\le n^{1-\epsilon}.
  \end{equation}
\end{proposition}

\begin{remark}This proposition, combined with (\ref{item:1}) of Proposition~\ref{prop:equiv}, immediately implies that
\begin{equation}\uc{c:1} w_n \le n^{1-\varepsilon}.\end{equation}\end{remark}
\begin{proof}
  \nc{c:10} The main step in the proof is to show that there exists a constant $\uc{c:10}>0$ such that for every $n\ge 1$,
 \begin{equation}
   \label{eq:34}
   \mathbb E (R_{2n}^+)\le (2-\uc{c:10}) \mathbb E (R_{n}^+).
 \end{equation}
 In order to compare $\mathbb E (R_{n}^+)$ with $\mathbb E (R_{2n}^+)$, it will be convenient to introduce the following more general variables. For $0\le m\le n$, define
\begin{equation}
  \label{eq:35}
  R_{m,n}^+:=\max\big\{\,0,\,\sup\{x\ge 0\:\text{s.t.}\: (-\infty,R_m^+] \times \{m\} \lr (x+R_m^+,n)\}\big\}.
\end{equation}
Note that $R_{0,n}^+=R_n^+$ for every $n\ge 0$. Before moving further, let us mention two other useful properties of these variables that follow from the definition. First, translation invariance and independence imply that
\begin{equation}
  \label{eq:36}
  \mathbb E(R_{m,n}^+)=\mathbb E(R_{n-m}^+).
\end{equation}
Furthermore, for every percolation configuration $\omega$, we have the following sub-additivity property
\begin{equation}
  \label{eq:37}
  R_{0,n}^+(\omega)\le R_{0,m}^+(\omega)+R_{m,n}^+(\omega).
\end{equation}
 Note that \eqref{eq:36} and \eqref{eq:37} already imply that for every $n$,
 \begin{equation}
   \label{eq:38}
   \mathbb E(R_{2n}^+) \le \mathbb E(R_{0,n}^+)+ \mathbb E(R_{n,2n}^+) =2\mathbb E(R_{n}^+).
 \end{equation}
 Hence, in order to prove \eqref{eq:34}, we need to show that the inequality above is not sharp. We do this by constructing an event on which $R_{2n}^+$ is significantly smaller than  $ R_{0,n}^+ +R_{n,2n}^+$.

Fix $n\ge 1$. Recall the definition of the event $E$ and let $F$ be the event (see Fig.~\ref{construction}) that
\begin{itemize}
\item $[\tfrac12 w_n,\tfrac32 w_n]\times[n,2n]$ and $[\tfrac52 w_n,\tfrac72w_n]\times[0,2n]$ are not crossed from left to right,
\item $[\tfrac12 w_n,\tfrac72 w_n]\times[n,2n]$ is not crossed vertically.
\end{itemize}
 The box-crossing property and the FKG inequality imply that $\bbP(F)\ge \uc{c:box_crossing}^3$. Since $E$ and $F$ depend on different sets of edges, independence and \eqref{eq:17} give
$$\bbP(E\cap F)\ge \tfrac14 \uc{c:box_crossing}^6.$$
Now, observe that when the event $E\cap F$ occurs, we have $R_{n}^+\ge 2w_n$, $R_{n,2n}^+=0$ and $R_{2n}^+\le \tfrac32w_n$. Hence,
\begin{equation*}
  \label{eq:42}
  \mathds 1_{E\cap F}R_{2n}^+\le \mathds 1_{E\cap F}( R_{n}^++R_{n,2n}^+ -\tfrac12w_n).
\end{equation*}
On the event $(E\cap F)^c$, the trivial bound provided by \eqref{eq:37} gives
\begin{equation*}
  \label{eq:41}
   \mathds 1_{(E\cap F)^c}R_{2n}^+\le \mathds 1_{(E\cap F)^c}( R_{n}^++R_{n,2n}^+).
\end{equation*}
Summing the two equations above and taking the expectation, we find
\begin{align*}
  \label{eq:43}
  \mathbb E (R_{2n}^+) &\le 2\mathbb E(R_{n}^+) -\tfrac12w_n \mathbb P(E\cap F)\le (2-\tfrac18\uc{c:box_crossing}^6 \uc{c:1}) \mathbb E(R_{n}^+).
\end{align*}
In the second inequality, we used the bound $w_n\ge \uc{c:1}\mathbb E(R_{n}^+)$ provided by Proposition~\ref{prop:equiv}. This finishes the proof of \eqref{eq:34}, which implies the statement of the proposition along the geometric sequence $n=2^k$. The general statement of \eqref{eq:9} follows by sub-additivity.
\end{proof}

\subsection{Proof of the main theorems}
\label{sec:proof-theor-refthm:2}

\begin{proof}[Proof of Theorem~\ref{thm:2}]
By Proposition~\ref{prop:equiv}, it is sufficient to get the similar bound for $w_n$.
The bounds then follows from Propositions~\ref{cor:polybound} and \ref{lem:sublinear}.
\end{proof}

\begin{proof}[Proof of Theorem~\ref{thm:1}]
The lower bound follows from \eqref{eq:47} and Proposition~\ref{cor:polybound}. We now focus on the upper bound.

First, the box-crossing property and the FKG inequality imply that the event $E_n$ defined (see Fig.~\ref{E_n}) by
\begin{itemize}
\item $[\tfrac12w_n,\tfrac32w_n]\times[0,2n]$ is not crossed from left to right,
\item $[-\tfrac32w_n,-\tfrac12w_n]\times[0,2n]$ is not crossed from right to left,
\item $[-\tfrac32w_n,\tfrac32w_n]\times[0,2n]$ is not crossed vertically,
\end{itemize}
satisfies
\begin{equation}\label{eq:222}
\bbP(E_n)\stackrel{(FKG)}\ge H(3w_n,n)V(w_n,3n)^2 \ge \uc{c:box_crossing}^3.
\end{equation}
Let $r\ge2$ be a large enough integer that we fix later and set $K:=\lfloor\log_r(n/2)\rfloor$.
For the event $0\rightarrow \ell_n$ to occur, none of the events $E_{r^k}$, $1\le k\le K$, should occur. Imagine for a moment that $w_{rn}> 3w_{n}$ for each $n$, then the events $E_{r^k}$, $1\le k\le K$, depend on different sets of edges, so that \eqref{eq:222} implies
\nc{a}$$\bbP(0\rightarrow \ell_n)\le \bbP\Big(\bigcap_{k=1}^{K} E_{r^k}^c\Big)=\prod_{k=1}^{K}(1-\bbP(E_{r^k}))\le (1-\uc{c:box_crossing}^3)^{K}\le n^{-\uc{a}}.$$
To conclude the proof, we therefore need to show that $w_{rn}\ge 3w_n$, or equivalently, by definition \eqref{eq:abca} of $w_{rn}$, that
$$H(\alpha 3w_n,\varepsilon r n)> V(3w_n,\alpha \varepsilon r n).$$
On the one hand, provided $r>1/\varepsilon$, monotonicity and the box-crossing property implies
$$H(\alpha 3w_n,\varepsilon r n)\ge H(3w_n,n)\ge \uc{c:box_crossing}.$$
On the other hand, if the box $[0,3w_n]\times[\alpha \varepsilon r n]$ is crossed vertically, then
$s=\lfloor r\alpha\varepsilon \rfloor$ disjoint boxes of width $3w_n$ and height $n$ must be crossed vertically (see Fig.~\ref{E_n}). Therefore,
$$V(3w_n,\alpha \varepsilon r n)\le V(3w_n,n)^s\le (1-\uc{c:box_crossing})^s.$$
Providing $ r$ large enough, we may guarantee that $(1-\uc{c:box_crossing})^s<\uc{c:box_crossing}$, and therefore $w_{rn}\ge 3w_n$ for every $n\ge1$.
\end{proof}

\begin{remark}
In order to obtain the slightly weaker bound
$$\bbP(0\rightarrow \ell_n)\ge \frac1{n^{(1-\varepsilon)/2}},$$
one may avoid the use of (1.8) in \cite{DurSchTan89} by simply combining the bound $w_n\le n^{1-\varepsilon}$ with \eqref{eq:47}.
\end{remark}
      \begin{figure}
    \label{E_n}
    \centering
 \includegraphics[width=1.00\textwidth]{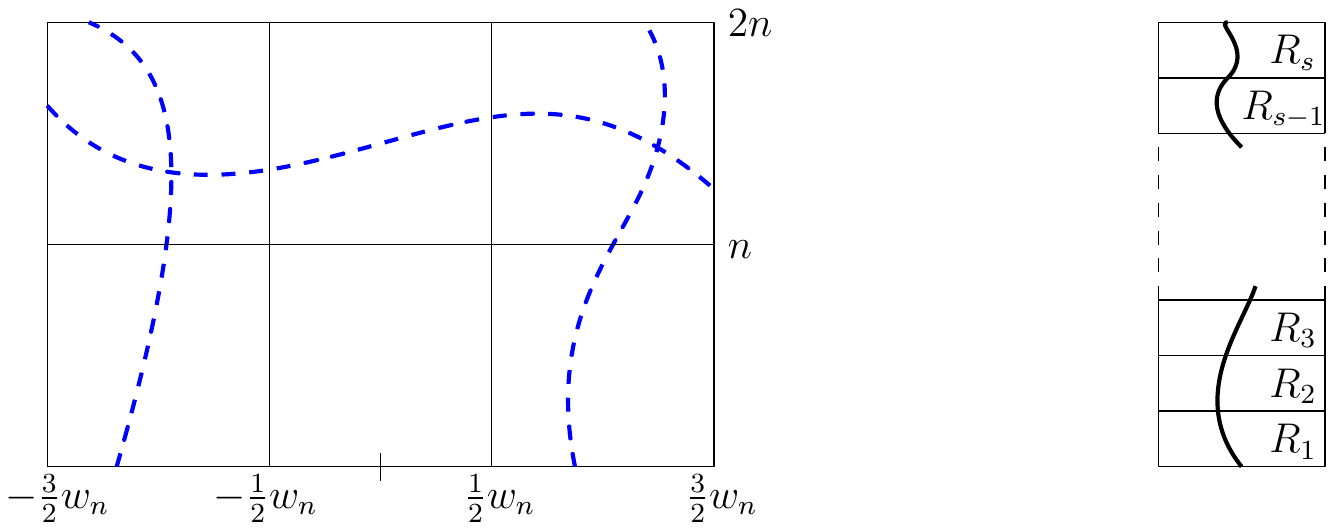}
    \caption{On the left, an illustration of the event $E_n$. Again, the blue dotted line denotes a dual path preventing the existence of an oriented path from $[-\tfrac12w_n,\tfrac12w_n]\times[0,n]$ to the outside of $[-\tfrac32w_n,\tfrac32w_n]\times[0,2n]$. On the right, if a box of width $3w_n$ and height $sn$ is crossed vertically, then $s$ rectangles of width $3w_n$ and height $n$ are crossed vertically.}
  \end{figure}
\bibliographystyle{amsalpha}
\bibliography{bibli}

\end{document}